\documentclass[12pt]{article}

\usepackage[normalem]{ulem}

\usepackage{amssymb}
\usepackage{times}
\usepackage[fleqn]{amsmath}
\usepackage{amsthm,amsmath,amsfonts,amssymb}
\usepackage[all]{xy}

\theoremstyle{plain}
\newtheorem{theorem}{Theorem}[section]
\newtheorem{lemma}[theorem]{Lemma}
\newtheorem{proposition}[theorem]{Proposition}

\theoremstyle{definition}

\newtheorem{example}{\it Example\/}
\newtheorem*{const}{\it Construction\/}

\def\real{\mathbb{R}}

\def\D{\mathcal{D}}
\def\E{\mathcal{E}}
\def\FF{\mathcal{F}}

\def\LL{\mathcal{L}}

\begin{document}
\title{De~Rham cohomology \\ of diffeological spaces and foliations\footnote{\MEC}}

\author {G.~Hector  \and E.~Mac\'{\i}as-Virg\'os \and E.~Sanmart\'{\i}n-Carb\'on}

\def\MEC{Partially supported by FEDER and Research Project MTM2008-05861 MICINN Spain.}

\date{}
\maketitle
%
%
\begin{abstract}
Let $(M,\FF)$ be a foliated manifold. We prove that there is a canonical isomorphism between the complex of base-like forms $\Omega^*_b(M,\FF)$ of the foliation and the ``De Rham complex" of the space of leaves  $M/\FF$ when considered as a ``diffeological" quotient. Consequently, the two corresponding cohomology groups $H^*_b(M,\FF)$ and $H^*(M/\FF)$ are isomorphic.

As an application we give a quick proof for the topological invariance of the base-like cohomology of Lie foliations (and Riemannian foliations) on compact manifolds.\\


\end{abstract}

\section{Introduction}
Diffeological spaces were introduced by J.-M.~Souriau in \cite{SOURIAU}
(and in a slightly different form by K.-T. Chen in
\cite{CHEN})
as a generalization of the notion of manifold. Spaces of maps and quotients of manifolds fit naturally into this category. Moreover, many
definitions from differential geometry can be extended to this setting. In particular, due to their contravariant nature, ``differential forms" and ``De Rham cohomology groups" extend to diffeological spaces in a canonical way
(for a  complete description of the De Rham calculus on diffeological spaces see
P.~Iglesias' book
\cite{IGLESIAS}).

In the context of Foliation Theory, the most interesting examples of diffeological spaces are the spaces of leaves $M/\FF$ of foliated manifolds, endowed with the quotient diffeology induced by the smooth structure of the ambient manifold (see section \ref{DIFFEOLOGY} for a precise definition).
It is then a natural problem to compare its De~Rham cohomology $H^*(M/\FF)$ with the base-like cohomology $H^*_b(M,\FF)$ which is a well known algebraic invariant of the foliation $\FF$ (see for example \cite{HECTOR,REINHART}).

In this paper we prove that the two theories coincide for any foliation $\FF$  (see Theorem \ref{MAIN}). As an application we obtain (Theorem \ref{APLI}) a simple proof of the topological invariance of the base-like cohomology for Lie foliations on compact manifolds. This result, which extends by standard techniques of algebraic topology to the larger family of Riemannian foliations, was first proved by A. El Kacimi and M. Nicolau in \cite{NICOLAUKACIMI}.

We thank F.~Alcalde, J.~A. \'Alvarez and R.~Wolak for several useful comments and remarks.

\section{Diffeological spaces}\label{DIFFEOLOGY}

In this first section we introduce some basic notions concerning diffeological spaces and their cohomology groups (see \cite{IGLESIAS}).

 \subsection{Diffeologies of class $C^r$}
 Let $X$ be a set. A map $\xymatrix{\real^n \supset U \ar[r]^\alpha &  X}$ defined on an open subset $U$ of some euclidean space will be called a {\em $n$-parametrization} of $X$. Now a {\em diffeology of class $C^r$}  on $X$ is a family $\D^r$ of parametrizations satisfying the following axioms:
\begin{enumerate}
\item[(1)] any constant parametrization of any dimension $n\geq 0$ belongs to $\D^r$,
\item[(2)] let $\xymatrix{\real^n \supset U \ar[r]^\alpha &  X}$ be a parametrization of $X$; if there exists an open cover $\{V_j\}$ of $U$ such that the restriction $\alpha_j$ of $\alpha$ to $V_j$ belongs to $\mathcal{D}^r$ for any $j$, then $\alpha$ belongs also to $\mathcal{D}^r$,
\item[(3)] for $\alpha \in \D^r$ and any $C^r$-map $\xymatrix{ \real^m \supset V \ar[r]^h & U}$, the composition $\alpha \circ h \in \D^r$.

\end{enumerate}
\noindent A set $X$ endowed with a diffeology $\D^r$ will be called a {\em diffeological space of class $C^r$}, any parametrization $\alpha \in \D^r$ being a {\em plot} of $X$.

\bigskip
\bigskip
The following additional definitions and observations will be relevant for the description of the category of diffeological spaces.

(a) It will be often convenient to define a diffeology on $X$ by means of a \textit{generating set}: for any set $\mathcal{G}$ of parametrizations of $X$ and any integer $r$, there exists a minimal diffeology of class $C^r$ containing $\mathcal{G}$. It will be called the \textit{$C^r$-diffeology generated by $\mathcal{G}$}.

(b) Let $f\colon X \to Y$ be a map from a diffeological space $(X, \D)$ to a set $Y$. The set $f(\mathcal{D})$ of parametrizations of type $f \circ \alpha$, where $\alpha \in \mathcal{D}$, generates a diffeology called \textit{the direct image of $\mathcal{D}$ by $f$}.

(c) Now we define the morphisms in the category of diffeological spaces: given two diffeological spaces $(X, \D^r)$ and $(Y , \E^r)$ of class $C^r$, a map $f\colon  X \to Y$ is a \textit{diffeological map of class $C^r$} (sometimes called \textit{differentiable map} in the literature) if $f(\mathcal{D}^r) \subset \mathcal{E}^r$.

(d) Finally notice that a diffeology of class $C^r$ generates a diffeology of class $C^s$ for any $s \leq r$. Thus we can speak about diffeological maps of class $C^s$ for spaces of class $C^r$.

\medskip
Indeed we will be mostly interested in diffeologies
of class $C^\infty$ also called \textit{smooth diffeologies} and diffeologies of class  $C^0$ also called \textit{topological diffeologies}.

\begin{example}\label{exple.manifold}{\it Manifold diffeologies.}\\
An atlas $\mathcal{V}$ of class $C^r$ on a manifold $M$ generates a $C^r$-diffeology which depends only on the $C^r$-structure defined by $\mathcal{V}$. It is a \textit{manifold diffeology} and consists of all $C^r$-parametrizations of $M$.
If $M$ and $N$ are two manifolds endowed with the corresponding
$C^r$-diffeologies, a map $f\colon M \to N$ is
$C^r$-diffeological if and only if it is
$C^r$-differentiable in the usual sense. In case $r=0$, such a map is just a continuous map.
\end{example}

\begin{example}\label{exple.quotient}{\it Quotient diffeologies.}\\
{Let $\pi\colon X \to X/\rho$ be the quotient map of a set $X$ by an equivalence relation $\rho$. If $\mathcal{D}$ is a diffeology on $X$, the direct image $\pi(\mathcal{D})$ generates a diffeology on $X/\rho$ which is called the \textit{quotient diffeology} associated to the relation $\rho$. One can define it as being the weakest diffeology on $X/\rho$ which makes $\pi$ a diffeological map.

If $(X,\rho)$ and $(X',\rho')$ are two spaces equipped with equivalence relations, a diffeological map $f\colon  (X,\rho) \to (X',\rho')$ which is compatible with $\rho$ and $\rho'$, induces a map $\bar f$ defining a commutative diagram :
$$
\xymatrix{X \ar[d] \ar[r]^f  &  X'\ar[d] \\
          X/\rho \ar[r]^{\bar f}  &  X'/\rho'}
$$
and $\bar f$ is diffeological with respect to the quotient diffeologies.

\medskip
There are two special cases of interest:

\smallskip
(a) If $\rho$ is the equivalence relation generated by a smooth action of a countable group $\Gamma$ on a smooth manifold $M$, then  $M/\rho$ is a smooth diffeological space with respect to the corresponding quotient diffeology, the latter being a manifold diffeology if $M/\rho$ is a manifold.

We will consider in particular the quotient of a connected Lie group $G$ by a dense countable subgroup $\Gamma$ which we will call a \textit{strongly homogeneous quotient.}

\smallskip
(b) If the classes of $\rho$ are the leaves of a smooth foliation $\FF$ on a manifold $M$, the quotient is the \textit{leaf-space} of $\FF$ denoted by $M/\FF$. It is again a smooth diffeological space.}
\end{example}

\subsection{De~Rham cohomology of smooth diffeological spaces}

Here we consider a smooth diffeological space $(X, \D)$. Let $\omega_\alpha$ be a $r$-form on the domain $U$ of the plot $\alpha$. A family  ${\omega }=\{{\omega }_\alpha\}_{\alpha \in
\D}$ indexed by the plots $\alpha \in \D$, is a \textit{De~Rham $r$-form} on $X$ if it fulfills the compatibility condition:
$$
{\omega }_{\alpha \circ h}=h^*{\omega }_\alpha
$$
for any smooth map $h\colon \real^m \supset V \to U$. The exterior differential $d \omega$ of $\omega$ is defined by $d \omega= \{d\omega_\alpha\}$ and verifies obviously the usual property $d \circ d  = 0$.\\

The differential complex $\Omega^*(X,\mathcal{D})$ of all De~Rham forms of any degree on $(X,\mathcal{D})$ is called the \textit{De~Rham complex} of the diffeological space $(X, \mathcal{D})$; its cohomology is the \textit{De~Rham cohomology group} $H^*(X,\mathcal{D})$ of $(X, \mathcal{D})$.

Any smooth diffeological map $f\colon  (X,\mathcal{D}) \to (Y,\mathcal{E})$ induces a homomorphism of differential complexes
$$
f^*\colon \Omega^*(X,\mathcal{D}) \leftarrow  \Omega^*(Y,\mathcal{E})
$$
by the formula $ (f^*{\omega })_\alpha ={\omega }_{f\circ \alpha}$,  $\alpha\in\D$. It goes over to cohomology with the usual functorial properties.

\bigskip
Now a smooth manifold $M$ is naturally equipped with two De~Rham cohomology groups: the usual one and the diffeological one as defined above. Fortunately, these two theories coincide as we show next. Indeed let $\Omega_{DR}(M)$ be the usual De~Rham complex of $M$, then the \textit{tautological map}
$$
\tau\colon \Omega_{DR}^*(M) \to \Omega^*(M,\mathcal{D})
$$
defined by $\tau({\eta}) =\{\alpha^*{\eta}\}_{\alpha \in \D}$ for any form $\eta \in \Omega_{DR}^*(M)$ is a morphism of differential complexes. Moreover we have

\begin{theorem} \label{VARIEDAD}- For any smooth manifold $M$, the tautological map $\tau$ is an isomorphism of differential complexes thus inducing an isomorphism
$$
\tau^* \colon H^*_{DR}(M) \to H^*(M,\mathcal{D}).
$$
\end{theorem}

\begin{proof}- Recall that the manifold diffeology of $M$ is generated by any locally finite smooth atlas $\{(V_j, \varphi_j)\}$ of $M$. Then if $\dim(M) = n$, a diffeological form $w \in \Omega^*(M,\mathcal{D})$ is completely determined by the  set of its components $w_j \in \Omega^*(U_j)$ on the open sets $U_j = \varphi_j(V_j) \subset \mathbb{R}^n$. By the compatibility condition of diffeological forms, the local forms $\eta_j = \varphi_j^*(w_j) \in \Omega^*(V_j)$ coincide on the overlaps of the sets $V_j$ thus define a global usual De Rham form $\eta$ on $M$. Setting $\eta = \sigma(\omega)$, we define a map
$$
\sigma\colon \Omega^*(M,\mathcal{D}) \to \Omega^*_{DR}(M)
$$
which obviously commutes with the differential and is an inverse for $\tau$.
\end{proof}

\bigskip
Note that this result has also been proved by P.~Iglesias in \cite{IGLESIAS}.\\
From now on we identify the two complexes at hand by means of the tautological map $\tau$ and denote by $\Omega^*(M)$ and $H^*(M)$ the De Rham complex and cohomology of the manifold $M$ whether in the usual or in the diffeological sense. Furthermore, we will denote by $\Omega^*(X)$ and $H^*(X)$ the De Rham complex and cohomology of any diffeological space $X$ without mention of the diffeology when there is no ambiguity.

\section{Main Theorem}

\subsection{Base-like cohomology of foliations}

Let $(M,\FF)$ be a foliated manifold. A differential form $\omega \in \Omega^r(M)$ is {\em base-like} for $\FF$ if for any vector
field $X$ tangent to $\FF$, we have
$$i_X\omega=0\quad\mathrm{and}\quad i_X d\omega=0. $$
In particular, the Lie derivative $\LL_X\omega$ of $\omega$ with respect to such a vector field $X$ vanishes and $\omega$ is preserved by the flow generated by $X$. The differential  of a base-like form is evidently base-like and we denote by
$$
\Omega^*_b (M,\FF) \subset \Omega^*(M)
$$
the subcomplex of base-like forms. Its cohomology $H^*_b(M,\FF)$ is the {\em base-like cohomology} of $(M,\FF)$ \cite{REINHART}. Our goal is to relate these complex and cohomology with the De Rham complex and cohomology of the diffeological quotient $M/\FF$.

\medskip
To do so, it will be convenient to introduce first a more appropriate description of base-like forms. An open subset $V \subset \mathbb{R}^n$ is an \textit{open cube (of dimension $m$)} if its closure $\bar V$ is homeomorphic to $[0,1]^m$.

\begin{const}-  Recall that any codimension $m$ foliation $(M,\mathcal{F})$ of class $C^r$ on a manifold $M$ of dimension $n$, can be defined by a \textit{foliated cocycle}
$$\mathcal{C} = (\{(V_i,f_i)\},  \{g_{ij}\})$$
with values in the pseudo-group $\mathcal{P}^m$ of local $C^r$-diffeomorphisms of $\mathbb{R}^m$ verifying the following:

i) the underlying covering $\mathcal{V}  =\{V_i\}_{i \in I}$ is a locally finite covering by open cubes,

ii) any $f_i\colon  V_i \to \mathbb{R}^m$ is a submersion over an open cube $Q_i$ all of whose fibers are open cubes of dimension $n-m$ called \textit{plaques},

iii) when $V_i \cap V_j \neq \emptyset$, the cocycle $\{g_{ij}\}$ determines a family of local $C^r$-diffeomorphisms $g_{ij}\colon  Q_j \to Q_i$ which generates the \textit{holonomy pseudo-group $\mathcal{H}$ of $\mathcal{F}$} acting on $Q = \coprod_i~ Q_i$.

\noindent Indeed we represent concretely each $Q_i$ as a local transverse cube to $\mathcal{F}$ cutting each plaque of $V_i$ in exactly one point and such that $Q_i \cap Q_j = \emptyset$ for $i \neq j$. Under these circumstances, we call $Q$ a \textit{total transversal to $\mathcal{F}$} and denote by $\chi\colon Q \to M$, the natural inclusion of $Q$ into $M$.
\end{const}

\bigskip

Next restricting forms of $M$ to $Q$, we obtain a homomorphism $\chi^*\colon \Omega^*(M) \to \Omega^*(Q)$
and it is routine to show the following:

\begin{proposition}- This homomorphism $\chi^*$ restricts as an isomorphism
$$
\chi^*_0\colon \Omega_b^*(M,\FF) \to \Omega_\mathcal{H}^*(Q)
$$
where $\Omega_\mathcal{H}^*(Q)$ is the complex of usual De Rham forms on $Q$ which are invariant by the holonomy pseudo-group $\mathcal{H}$.
\end{proposition}

\subsection{Comparison theorem}

We are now in position to state and prove our main theorem. With the notations of the previous section, let $(Q,\mathcal{H})$ be the holonomy pseudo-group of the foliated manifold $(M,\FF)$ and let $\rho$ be the associated equivalence relation. We have a commutative square:
$$
\xymatrix{Q \ar[d]^{\pi} \ar[r]^\chi &  M\ar[d]^p \\
          Q/\rho \ar[r]^{\bar \chi}  &  M/\FF}
$$
where $Q$ and $M$ are endowed with their manifold diffeologies, $Q/\rho$ and $M/\FF$ with the corresponding quotient diffeologies and $\bar \chi$ is the map induced  by the inclusion $\chi$ of $Q$ into $M$.

\begin{lemma}- For any foliated manifold $(M,\FF)$, the map $\bar \chi$ is a smooth diffeological isomorphism.
\end{lemma}

\begin{proof}- First note that $\chi$ and the three maps $\pi$, $p$ and $\bar \chi$ are smooth by definition of the quotient diffeologies. Next notice that $\bar \chi$ is bijective and thus it just remains  to show that its inverse is smooth.

To do so recall that the manifold diffeology of $M$ can be generated by the set $\mathcal{G}$ of smooth parametrizations whose image is contained in some element $V_j$ of the covering $\mathcal{V}$. Then $p(\mathcal{G})$ generates the quotient diffeology of $M/\FF$ and for any plot $\gamma \in p(\mathcal{G})$, there exists an index $j$ and a plot $\alpha$ with range in $V_j$ such that $\gamma = p \circ \alpha$. Let $f_j\colon V_j \to Q_j$ be the local projection provided by the foliated cocycle $\mathcal{C}$, then $f_j \circ \alpha$ is a plot in $Q_j$ such that
$$
\bar \chi \circ \pi \circ f_j \circ\alpha = p \circ \alpha = \gamma
$$
showing that $(\bar \chi)^{-1}[p(\mathcal{G})]$ is contained in the diffeology of $Q/\rho$. This latter condition means that $(\bar \chi)^{-1}$ is smooth; the proof is complete.
\end{proof}

In the sequel we will identify base-like forms and forms on $Q/\rho$. Next we focus on the projection $\pi\colon Q \to Q/\rho$. The crucial point in the procedure will be the following geometrical observation.

\begin{lemma}\label{lem.crucial}- Let $\alpha, \beta \colon U \to Q$ be two plots of the transversal $Q$. If $\pi \circ \alpha = \pi \circ \beta$, there exists a countable family of open sets $W_s \subset U$ and elements $h_s \in \mathcal{H}$ such that:

i) $\bigcup_s W_s$ is dense in $U$,

ii) $\beta = h_s \circ \alpha$ when restricted to $W_s$.
\end{lemma}

\begin{proof}- For any element $h \in \mathcal{H}$ there exists $X_h \subset U$, possibly empty, such that $\beta(x) = h \circ \alpha(x)$ for any $x \in X_h$ and this set is a closed subset of $U$ by continuity. On the other hand, condition $\pi \circ \alpha = \pi \circ \beta$ implies that for any point $x \in U$ there exists an element $h_x \in \mathcal{H}$ such that $\beta(x) = (h_x \circ \alpha)(x)$ and consequently  $\bigcup X_h$ covers $U$. But the set of maximal elements of $\mathcal{H}$ is countable and therefore by Baire theory there exists an element $h_1$ such that $X_{h_1}$ has non empty interior $W_1$.

We apply the same argument to $U \backslash X_{h_1}$ and repeating infinitely many times the construction if necessary, we obtain the wanted sequence.
\end{proof}

\begin{lemma}- The homomorphism $\pi^* \colon \Omega^*(Q) \leftarrow \Omega^*(Q/\rho)$
is an isomorphism of $\Omega^*(Q/\rho)$ onto the sub-complex $\Omega_\mathcal{H}^*(Q)$ of forms on $Q$ invariant by $\mathcal{H}$.

\end{lemma}

\begin{proof}- The proof is similar to that of theorem \ref{VARIEDAD}.

First we note that the image of $\pi^*$ is contained in  $\Omega_\mathcal{H}^*(Q)$ because $\pi\circ h =\pi$ for any $h\in \mathcal{H}$. Thus it remains to construct an inverse $\sigma$  of  $\pi^*$  on the $\mathcal{H}$-invariant forms as follows.

Fix $v \in \Omega^*_\mathcal{H}(Q)$ and let $\alpha, \beta: U \to Q$ be two plots of $Q$ such that $\pi \circ \alpha = \pi \circ \beta = \gamma$ a plot of $Q/\rho$. As $v$ is $\mathcal{H}$-invariant, we can consider the family of local diffeomorphisms $h_s$ and open sets $W_s$ provided by lemma \ref{lem.crucial}. Because any $h_s$ belongs to $\mathcal{H}$ and $v$ is $\mathcal{H}$-invariant, we get
$$
\beta^*(v) = (h_s \circ \alpha)^*(v)= \alpha^*(h_s^* v) = \alpha^*(v) \quad \mathrm{on\ }W_s,
$$
and as $\bigcup_s W_s$ is dense in $U$, it follows that $\alpha^*(v) = \beta^*(v)$ and so this form depends only on $\gamma$; we denote it by $u_\gamma$.

Next recall that the diffeology of $Q/\rho$ is generated by the set $\mathcal{G}$ of plots $\gamma\colon U \to Q/\rho$ for which there exists a lift $\alpha\colon U \to Q$ verifying $\gamma = \pi \circ \alpha$. Then we define a homomorphism $\sigma \colon \Omega^*_\mathcal{H}(Q) \to \Omega^*(Q/\rho)$ by setting $\sigma(v)_\gamma = u_\gamma$, $\gamma\in\mathcal{G}$. It is easy to check that in this way we obtain a well defined  diffeological form $u\in \Omega^*(Q/\rho)$.
It verifies $\pi^* u = v$ because by definition 
$$(\pi^*u)_\alpha= u_{\pi\circ\alpha}= u_\gamma=\alpha^*(v)$$ on the   domain  of   $\gamma = \pi \circ \alpha$. The proof is complete.
\end{proof}

To conclude, we note that $p^* = (\chi_0^*)^{-1} \circ \pi^* \circ \bar \chi^*$ which leads to our \textit{comparison theorem}:

\begin{theorem}\label{MAIN}- For any foliation $\FF$, the projection $p\colon  M \to M/\FF$ induces an isomorphism
$$
p^*\colon \Omega^*_b(M,\FF) \leftarrow \Omega^*(M/\FF).
$$
Consequently the De~Rham cohomology $H^*(M/\FF)$ of $M/\FF$ is isomorphic to the base-like cohomology $H^*_b(M,\FF)$ of $\FF$.
\end{theorem}

\section{$C^0$-invariance of the De Rham cohomology}

It is a well known fact that De Rham cohomology is a topological invariant for smooth manifolds. Here we investigate the corresponding question for smooth diffeological spaces focusing on a restricted class of such spaces. The goal is to recover, as an application of our theorem \ref{MAIN}, the topological invariance of the base-like cohomology for Riemannian foliations on compact manifolds established by A.~El~Kacimi and M.~Nicolau in \cite{NICOLAUKACIMI}. Our argument will be very quick and straightforward showing the power and simplicity of the diffeological approach.

\subsection{The case of strongly homogeneous quotients}

Consider a strongly homogeneous diffeological space $G/\Gamma$ as defined in example \ref{exple.quotient} above. If $\tilde \Gamma = q^{-1}(\Gamma)$, the universal covering  $q\colon \tilde G \to G$ induces a smooth diffeological map $\tilde q\colon \tilde G/\tilde \Gamma \to G/\Gamma$, which is a smooth isomorphism of diffeological spaces. Thus without loss of generality, we may always assume that $G$ is simply connected (and connected).

The group of smooth diffeological isomorphisms of $G/\Gamma$ was explicitly computed in \cite{MACHECT}.
Here we focus on continuous diffeological isomorphisms and want to prove the following result:

\begin{lemma}\label{lem.strong}- Let $G/\Gamma$ and $G'/\Gamma'$ be two strongly homogeneous diffeological spaces. Then any $C^0$-diffeological isomorphism
$$\varphi\colon G/\Gamma \to G'/\Gamma'$$
is indeed a $C^\infty$-isomorphism and the two Lie groups $G$ and $G'$ are isomorphic.
\end{lemma}

\begin{proof}- The proof uses the theory of covering spaces in the diffeological category  developed in \cite{IGLESIASTESIS} which is formally the same as in the topological category. It follows the lines of an argument first used in \cite[p.~252]{MACHECT}.

(a) As $G$ and $G'$ are simply connected, they are the universal coverings (in the category of diffeological spaces) of $G/\Gamma$ and $G'/\Gamma'$ respectively. Thus the map $\varphi$ lifts as a $C^0$-diffeological isomorphism, that is we have a commutative diagram
$$
\xymatrix{ G \ar[d] \ar[r]^\phi & G' \ar[d] \\
   G/\Gamma \ar[r]^\varphi & G'/\Gamma'. }
$$
where the lift $\phi$ is indeed a homeomorphism. And our goal being to prove that $\varphi$ is a smooth diffeological isomorphism it will be enough to show that the lift $\phi$ is smooth in the usual (or equivalently diffeological) sense. To do so, we will show that $e$ being the identity of $G$ and $\nu = \phi(e)^{-1}$, the map $\psi = L_\nu \circ \phi$  is a continuous and thus a smooth Lie group isomorphism.

(b) Indeed, fix an element $\gamma\in\Gamma$. For any $g\in G$ the two elements $g$ and $g\gamma$ are equivalent by $\Gamma$ and consequently $\phi(g)$ and $\phi(g\gamma)$ are equivalent by $\Gamma'$ because $\phi$ is equivariant with respect to the right actions of $\Gamma$ and $\Gamma'$. This means also that for any $g$
$$
\Phi(g) := \phi(g)^{-1}\cdot \phi(g\gamma)\in\Gamma^\prime
$$
and $G$ being connected and $\Gamma'$ being totally disconnected, the continuous map $\Phi\colon G \to \Gamma'$ is constant which implies
$$
\phi(g)^{-1}\cdot \phi(g\gamma)= \phi(e)^{-1}\cdot \phi(\gamma)$$
thus
$$\phi(g\gamma) = \phi(g) \cdot \phi(e)^{-1}\cdot \phi(\gamma),$$
for any $\gamma \in \Gamma$ and $g \in G$. The latter relation can be rewritten as
$$
 L_\nu[\phi(g\gamma)] = L_\nu[\phi(g)] \cdot L_\nu[\phi(\gamma)]  $$or
 $$\psi(g\gamma) = \psi(g)\cdot \psi(\gamma).$$

(c) Consequently, we see that the $\psi$ which is continuous by definition restricts to a continuous group homomorphism from $\Gamma$ to $\Gamma'$. And $\Gamma$ being dense in $G$ it follows easily that $\psi\colon G \to G'$ is a homomorphism thus a continuous group isomorphism and finally a smooth Lie group isomorphism. This means that $G$ is isomorphic to $G'$ and $\varphi$ is smooth.
\end{proof}

\subsection{$C^0$-invariance of the base-like cohomology of Riemannian foliations}

Let $\mathfrak{G}$ be the Lie algebra of a connected (and simply connected) Lie group $G$. A foliation $\FF$ on a closed manifold $M$ is a \textit{$\mathfrak{G}$-Lie foliation} if it is defined by a non vanishing $\mathfrak{G}$-valued $1$-form $\omega$ verifying the Maurer-Cartan equation
$$
\omega + \frac{1}{2}[\omega,\omega] = 0.
$$
These foliations have been introduced first by E. Fedida  who gives the following nice description  \cite {FEDIDA,MACIAS}.

(a) A $\mathfrak{G}$-Lie foliation $(M,\FF)$ lifts to the universal covering $p\colon  \tilde M \to M$ as a locally trivial fibration $\tilde \FF$ over the Lie group $G$. This fibration is preserved by the action of the fundamental group $\pi_1(M)$ and there is a representation $h\colon  \pi_1(M) \to G$ whose image $\Gamma$ is a finitely generated subgroup of $G$ called the \textit{Global Holonomy group of $\FF$}.

(b) In particular the leaf-space $M/\FF$ identifies as a diffeological space with the strongly homogeneous space $G/\Gamma$ and we get the following commutative diagram of diffeological spaces and maps:
$$
\xymatrix{ \tilde M \ar[d]^p \ar[r]^D & G \ar[d]^{\pi} \\
                M \ar[r]^{\bar D}   & G/\Gamma,}
$$
where $\bar D$ and $\pi$ are the canonical projections. Moreover the foliation $\FF$ is minimal if and only if $\Gamma$ is dense in $G$.

\bigskip

As a trivial consequence of lemma \ref{lem.strong}, we get the following invariance result:

\begin{theorem}- Let $(M,\FF)$ and $(M,\FF')$ be two minimal Lie foliations conjugate by a foliated homeomorphism $\psi$.  Consider the commutative diagram
$$
\xymatrix{ M \ar[d] \ar[r]^\psi & M' \ar[d] \\
G/\Gamma \ar[r]^\varphi & G'/\Gamma' }
$$
where $G/\Gamma$ and $G'/\Gamma'$ are the corresponding leaf-spaces and $\varphi$ is the $C^0$-diffeological isomorphism induced by $\psi$.

Then $\varphi$ and $\psi$ are both smooth. In particular $\varphi$ induces   isomorphisms between the base-like complexes $\Omega^*_b(M,\FF)$ and $\Omega^*_b(M',\FF')$ and between the base-like cohomology groups $H^*_b(M,\FF)$ and $H^*_b(M',\FF')$.
\end{theorem}

Now applying standard techniques from algebraic topology and using the structure theory of Molino \cite{MOLINO} one recovers the theorem of  El~Kacimi-Nicolau \cite{NICOLAUKACIMI}:

\begin{theorem}\label{APLI}- The base-like cohomology is a topological invariant for Riemannian foliations.
\end{theorem}

\bigskip

 {G.~Hector\\
 Institut C. Jordan - UMR CNRS 5028\\
 Math\'ematiques - Universit\'e Lyon 1\\
 43 Bd du 11 Novembre 1918\\
 69622 Villeurbanne-Cedex\\
 {\tt gilb.hector@gmail.com}\\

 E.~Mac\'{\i}as-Virg\'os\\
Department of Geometry and Topology,\\
University of Santiago de Compostela,\\
Avda. Lope de Marzoa s/n. Campus Sur. \\15782-Santiago de Compostela, Spain \\
{\tt quique.macias@usc.es\\
http://www.usc.es/imat/quique} \\

E.~Sanmart\'{\i}n-Carb\'on\\
Department of Mathematics,\\ University of Vigo,\\
F. CC. EE., \\
R\'ua Leonardo da Vinci, Campus Lagoas-Marcosende.
\\36310-Vigo, Spain\\
{\tt esanmart@uvigo.es}
}

\end{document}